\documentclass[11pt]{article}
\usepackage{graphicx,latexsym}
\usepackage{multicol,wrapfig,amsmath,amssymb,amsthm,bm,mathrsfs,braket}
\usepackage[numbers]{natbib}

\newtheorem{Theorem}{Theorem}[section]
\newtheorem{Lemma}{Lemma}[section]

\newtheorem{Prop}{Proposition}[section]

\newtheorem{Example}{Example}[section]

\allowdisplaybreaks

\title{A lower bound for the Graver complexity
of the incidence matrix of a complete bipartite graph}

\author{Taisei Kudo%
\thanks{Department of 
Mathematical Informatics, 
Graduate School of Information Science and Technology, 
University of Tokyo.} \ and \ 
Akimichi Takemura\footnotemark[1]\ \thanks{JST, CREST}}
\date{February 2011}

\begin{document}
\maketitle

\begin{abstract}
We give an exponential lower bound for
the Graver complexity of the incidence matrix of a complete bipartite graph 
of arbitrary size. 
Our result is a generalization of the result by \citet{berstein09}
for the complete bipartite graph $K_{3,r}$, $r\ge 3$.
\end{abstract}


\section{Introduction and the main result}
The Graver complexity of an integer matrix is currently actively
investigated for its importance to integer programming, algebraic
statistics and other applications (\cite{berstein09}, \cite{loera08}, \cite{hosten07}, \cite{aoki03}).  In particular, from the 
universality  of the three-way transportation program 
to general integer programs (\citet{loera06}), the Graver complexity 
of the incidence matrix of the  complete bipartite graph $K_{3,r}$ is 
particularly important. \citet{berstein09} proved that
the Graver complexity $g(r)$ for the incidence matrix of $K_{3,r}$, $r\ge 3$, 
is bounded below as $g(r)=\Omega(2^r)$, where $g(r)\ge 17\cdot 2^{r-3} -7$.
It is a natural question to generalize this result to the complete
bipartite graph $K_{t,r}$ of arbitrary size $t,r$.  
We prove that the Graver complexity for $K_{t,r}$ 
is $\Omega((t-1)^r)$, where $t\ge 4$ is fixed and $r$ diverges to infinity.
For proving our result, we employ double induction on $r$ and $t$ starting
from the result of \cite{berstein09}.

Let $A_{t,r}$ 
denote the incidence matrix of the
complete bipartite graph $K_{t,r}$ and let 
$g(A_{t,r})$ denote its Graver complexity.
Here we state our main theorem.
Relevant notations and definitions will be given in the next section.

\begin{Theorem}
\label{thm:main}
The Graver complexity of $A_{t,r}$ for any $4\le t\le r$ 
is bounded from below as
\begin{equation}\nonumber
 g(A_{t,r})\geq (t-1)^{r-t}(b_t+\frac{1}{t-2})-\frac{1}{t-2},  
\end{equation}
where
\begin{equation}\nonumber
 b_t=
(t-2)!\left(15+\sum_{i=1}^{t-4}\frac{i+4}{(i+2)!}\right).
\end{equation}
\end{Theorem}

We give a proof of this theorem in Section 
\ref{sec:proof} after giving necessary definitions
and reviewing relevant known results
in Section \ref{sec:preliminaries}.  
We conclude the paper with some discussion in
Section \ref{sec:discussion}.

\section{Preliminaries}
\label{sec:preliminaries}
In this section we summarize our notation and review relevant known
results on the Graver complexity following \citet{berstein09}.

The integer kernel of an $s\times t$ integer matrix  
$A$ is denoted by $\ker_{\mathbb Z}(A)=\{x\in{\mathbb Z}^t \mid
Ax=0\}$.  Define a partial order $\sqsubseteq$ on 
${\mathbb Z}^t$, which extends the coordinate-wise order $\leq$
on ${\mathbb Z}_+^t$,  as follows: 
For two vectors $u,v\in {\mathbb Z}^t$,
$u \sqsubseteq v$ if 
$|u_i|\leq |v_i|$ and 
$u_iv_i\geq 0$
for $i=1,\dots,t$.
The {\it Graver basis} ${\mathcal G}(A)$ of $A$ is the finite set 
of $\sqsubseteq$-minimal elements in the set $\ker_{\mathbb Z}(A)\setminus \{0\}$.

For any  fixed positive integer $h$, write an $ht$-dimensional
integer vector $x\in{\mathbb Z}^{ht}$ as  $x=(x^1,\dots,x^h)$ with each block 
$x^i$ belonging to ${\mathbb Z}^t$.
The {\it type} of $x=(x^1,\dots,x^h)$
is the number 
${\rm type}(x):=\#(\{i\mid x^i\neq 0\})$
of nonzero blocks of $x$.
The {\it $h$-th Lawrence lifting} of an 
$s \times t$ matrix $A$ is the following 
$(t+hs) \times ht$ matrix, with $I_t$ denoting 
the $t \times t$ identity matrix:
\begin{equation}
A^{(h)}:=
 \left(
\begin{array}{ccccc}
 A & 0 & 0 & \dots & 0 \\
 0 & A & 0 & \dots & 0 \\
 \vdots & \vdots & \ddots & \vdots & \vdots \\
 0 & 0 & 0 & \dots & A \\
 I_t & I_t & I_t & \dots & I_t \\
\end{array}
\right). 
\end{equation}
The {\it Graver complexity} of $A$
is defined as
\begin{equation}
  {g}(A)=\sup\left(\{0\}\cup\Set{{\rm type}(x) | x\in
\bigcup_{h\geq 1}{\mathcal G}\left(A^{(h)}\right)}\right).
\end{equation}

Let ${\mathcal G}({\mathcal G}(A))$ denote the Graver basis of
a matrix whose columns are the elements of 
${\mathcal G}(A)$ ordered arbitrarily.  The following result
shows that the Graver complexity of $A$ is determined by
${\mathcal G}({\mathcal G}(A))$. 
\begin{Prop}
\label{prop:graver-graver}
\cite{santos03} \ 
The Graver complexity of $A$ satisfies 
\begin{equation}\nonumber
 {g}(A)={\rm max}\{ \| x \|_1:
x\in{\mathcal G}({\mathcal G}(A)) \},
\end{equation}
where $\|\cdot \|_1$ denotes the 1-norm of a vector.
\end{Prop}

A {\it circuit} of an integer matrix $A$ 
is a nonzero integer vector $x\in\ker_{\mathbb Z}(A)$,
that has inclusion-minimal support with respect to $\ker_{\mathbb
Z}(A)$,
and whose nonzero entries are relatively prime. 
Let ${\mathcal C}(A)$ denote the set of circuits of a matrix $A$. 
Then ${\mathcal C}(A)\subseteq{\mathcal G}(A)$ (cf.\ \cite{MR1363949}).
An integer
relation $h=(h_1,\dots, h_k)$ 
on integer vectors $v^1,\dots,v^k \in {\mathbb Z}^t$ 
\[
0=h_1 v^1 + \dots + h_k v^k
\]
is {\it primitive} if  
$h_1,\dots,h_k$ are relatively prime positive integers and
no $k-1$ of the $\{v^i\}_{i=1}^k$ satisfy any nontrivial linear relation.
By ${\mathcal C}(A)\subseteq {\mathcal G}(A)$ and Proposition
\ref{prop:graver-graver} we have the following result.

\begin{Prop}
\label{prop23}
\cite{berstein09} \ 
Suppose that $h$ is a primitive relation on some set of circuits 
$\{x^i\}_{i=1}^k$ of an 
integer matrix $A$. 
Then the Graver complexity of $A$ 
satisfies $g(A)\geq \sum_{i=1}^k h_i$.
\end{Prop}

In this paper we consider the Graver complexity of the incidence
matrix $A_{t,r}$  for the complete bipartite graph $K_{t,r}$.
Let ${\bf 1}_t=(1,1,\dots,1)$ denote the $1\times t$ matrix 
consisting of $1$'s.  Then the $r$-th Lawrence lifting $A_{t,r}={\bf 1}_t^{(r)}$
of ${\bf 1}_t$ is the incidence matrix of $K_{t,r}$.  
In algebraic statistics, $A_{t,r}$ is the design
matrix specifying the row sums and the column sums of a two-way contingency
table.  
Another Lawrence lifting $({\bf 1}_t^{(r)})^{(h)}$ of ${\bf 1}_t^{(r)}$
is the design matrix for no-three-factor interaction model for  
$t\times r\times h$ three-way contingency tables (\cite{aoki03}, \cite{hosten07}). 
It is also the coefficient matrix
for the three-way transportation program.  The Graver complexity 
$g(A_{t,r})=g({\bf 1}_t^{(r)})$ gives the bound of complexity of the Graver
basis for the toric ideal associated with the 
no-three-factor interaction model for
$t\times r\times h$ three-way contingency tables as $h\rightarrow\infty$.

We employ below the following notation, where $t,r$ are positive integers.
Let 
\begin{equation}
 V:=\{v_1,\dots,v_t\}, \ U:=\{u_1,\dots,u_r\}.
\end{equation}
Then $V\oplus U$ and $V\times U$ denote 
the set of vertices and the set of edges of the complete bipartite graph
$K_{t,r}$, respectively. They 
index the rows and the columns of the incidence matrix $A_{t,r}$ of  $K_{t,r}$.
Here we explain interpretations of a circuit of $A_{t,r}$ 
referring to \cite{berstein09}.
We interpret each vector $x\in{\mathbb Z}^{V\times U}$ as:
\begin{enumerate}
 \item an integer valued function on the set of edges $V \times U$;
 \item a $t\times r$ matrix with its rows and columns indexed
 by V and U. 
\end{enumerate}
With these interpretations, $x$ is in ${\mathcal C}(A_{t,r})$ if and
only if:
\begin{enumerate}
 \item as a function on $V\times U$ with the following properties: 
  its support is a circuit of $K_{t,r}$, 
  along which its values $\pm 1$  alternate. It can be expressed 
 by the sequence
$(v_{i_1},u_{i_1},v_{i_2},u_{i_2},\dots,v_{i_l},u_{i_l})$
 of vertices of the 
circuit of $K_{t,r}$ on which it is supported, with the convention
that its value is $+1$ on the first edge $(v_{i_1},u_{i_1})$.
 \item as a nonzero matrix with the following properties: its elements are $0,\pm 1$, its row sums and columns sums are zeros, 
and it has an inclusion-minimal support with respect to 
these properties.
 
\end{enumerate}

The following example is the base case for our inductive argument
for the lower bound of the Graver complexity.

\begin{table}[hbtp]
\caption{Circuits for $A_{3,4}$ in a $3\times 4$ matrix form}
\begin{center}
\begin{tabular}{c|cccc|c} 
 & $u_1$ & $u_2$ & $u_3$ & $u_4$ &  \\\hline
 & $0$ & $0$ & $-1$ & $1$ & $v_1$ \\
$x^1=(v_1,u_4,v_3,u_2,v_2,u_3)=$ & $0$ & $-1$ & $1$ & $0$ & $v_2$ \\
 & $0$ & $1$ & $0$ & $-1$ & $v_3$ \\\hline
 & $-1$ & $1$ & $0$ & $0$ & $v_1$ \\
$x^2=(v_1,u_2,v_3,u_3,v_2,u_1)=$ & $1$ & $0$ & $-1$ & $0$ & $v_2$ \\
 & $0$ & $-1$ & $1$ & $0$ & $v_3$ \\\hline
 & $0$ & $-1$ & $0$ & $1$ & $v_1$ \\
$x^3=(v_1,u_4,v_2,u_1,v_3,u_2)=$ & $1$ & $0$ & $0$ & $-1$ & $v_2$ \\
 & $-1$ & $1$ & $0$ & $0$ & $v_3$ \\\hline
 & $-1$ & $0$ & $0$ & $1$ & $v_1$ \\
$x^4=(v_1,u_4,v_2,u_2,v_3,u_1)=$ & $0$ & $1$ & $0$ & $-1$ & $v_2$ \\
 & $1$ & $-1$ & $0$ & $0$ & $v_3$ \\\hline
 & $1$ & $0$ & $-1$ & $0$ & $v_1$ \\
$x^5=(v_1,u_1,v_2,u_2,v_3,u_3)=$ & $-1$ & $1$ & $0$ & $0$ & $v_2$ \\
 & $0$ & $-1$ & $1$ & $0$ & $v_3$ \\\hline
 & $0$ & $-1$ & $1$ & $0$ & $v_1$ \\
$x^6=(v_1,u_3,v_2,u_4,v_3,u_2)=$ & $0$ & $0$ & $-1$ & $1$ & $v_2$ \\
 & $0$ & $1$ & $0$ & $-1$ & $v_3$ \\\hline
 & $0$ & $1$ & $0$ & $-1$ & $v_1$ \\
$x^7=(v_1,u_2,v_2,u_3,v_3,u_4)=$ & $0$ & $-1$ & $1$ & $0$ & $v_2$ \\
 & $0$ & $0$ & $-1$ & $1$ & $v_3$ \\\hline
\end{tabular}
\end{center}
\label{sevent}
\end{table}

\begin{Example}\label{seven}\cite{berstein09} \ 
Let
$t=3$ and  
$r=4$.
Consider seven circuits $\{x^i\}_{i=1}^7$ in Table \ref{sevent} (written in a $3\times 4$ 
matrix form) of $A_{3,4}=(1,1,1)^{(4)}$.
They satisfy a primitive relation
\begin{equation}\nonumber
 x^1+2x^2+3x^3+3x^4+5x^5+6x^6+7x^7=0.
\end{equation}
Therefore from Proposition \ref{prop23}
\begin{equation}\nonumber
{g}(A_{3,4})\geq
1+2+3+3+5+6+7=27.
\end{equation}
\end{Example}

\section{Proof of the main theorem}
\label{sec:proof}
In this section we give a proof of our main theorem.
Our proof is based on  recursive 
construction of  primitive relations  for circuits of $A_{t,r}$.
We need recursions for $t$ and $r$, separately.  In Lemma
\ref{ttot1} we give a recursion for $t$ and in 
Lemma \ref{rtor1} we give a recursion for $r$.

\begin{Lemma}\label{ttot1}
Let $t\geq 4$. 
Suppose that there are circuits 
$\{x^i\}_{i=1}^k$ of
 $A_{t,t+1}={\bf 1}_t^{(t+1)}$ 
admitting a primitive relation $h$, where the $k$-th circuit and the $k$-th
coefficient are
\begin{align}\nonumber
 x^k&=(v_1,u_2,v_2,u_3,\dots,v_t,u_{t+1}),\\ \nonumber
h_k&=1. 
\end{align}
Then there are circuits $\{\bar{x}^i\}_{i=1}^{k+t}$
of $A_{t+1,t+1}={\bf 1}_{t+1}^{(t+1)}$
admitting a primitive relation 
$\bar{h}$, where the $(k+t)$-th circuit and the $(k+t)$-th coefficient are
\begin{align}\nonumber
 \bar{x}^{k+t}&=(v_1,u_1,v_2,u_2,\dots,v_{t+1},u_{t+1}),\\ \nonumber
\bar{h}_{k+t}&=1. 
\end{align}
\end{Lemma}

\begin{proof}
Using the natural embedding of 
 $K_{t,t+1}$ into $K_{t+1,t+1}$,
we can interpret circuits of the former also as circuits of the latter. 
Put
\begin{equation}\nonumber
 y^i=x^i, \ \forall i=1,\dots,k-1,
\end{equation}
and define 
\begin{align}\nonumber
 y^{k+j-1}&=(v_1,u_j,v_{t+1},u_{j+1}) , \ \forall j=1,\dots,t, \\ \nonumber
 y^{k+t}&=(v_1,u_2,v_2,u_3,\dots,v_{t},u_{t+1},v_{t+1},u_{1}). 
\end{align}
Table \ref{ttot1mat} displays 
$\{y^{k+j-1}\}_{j=1}^{t+1}$ as matrices, where
\begin{equation}\nonumber
 p=
(1,0,\dots,0,-1)^{\top}
\in{\mathbb Z}^{t+1}. 
\end{equation}
Blank entries are zeros.
Note that these circuits satisfy 
$\sum_{j=1}^{t+1}y^{k+j-1}=x^k$. 

\begin{table}[htbp]
\caption{Circuits for recursion on $t$}
\begin{center}
\begin{tabular}{c|cccccc|} 
 & $u_1$ & $u_2$ & $u_3$ & $\dots$ & $u_{t}$ & $u_{t+1}$ \\\hline
$y^{k}=(v_1,u_{1},v_{t+1},u_{2})=$ & $p$ & $-p$ &  &  &  &  \\\hline
$y^{k+1}=(v_1,u_{2},v_{t+1},u_{3})=$ &  & $p$ & $-p$ &  &  &  \\\hline
$\vdots$ &  &  & $\ddots$  & $\ddots$ &  &  \\\hline
$y^{k+t-1}=(v_1,u_{t},v_{t+1},u_{t+1})=$ &  &  &  &   $p$ & $-p$ & \\\hline
$y^{k+t-1}=(v_1,u_{t},v_{t+1},u_{t+1})=$ &  &  &  &  & $p$ & $-p$ \\\hline
 & $-1$ & $1$ &  &  &  &  \\
 &  & $-1$ & $1$ &  &  &  \\
$y^{k+t}=(v_1,u_2,v_2,u_3,\dots,v_{t},u_{t+1},v_{t+1},u_{1})= $
 &  &  & $\ddots$ & $\ddots$ &  &  \\
 &  &  &  &   $-1$ & $1$ &\\
 &  &  &  &  & $-1$ & $1$ \\
 & $1$ &  &  &  &  & $-1$ \\\hline
\end{tabular}
\end{center}
\label{ttot1mat}
\end{table}

Suppose that $\bar{h}\in{\mathbb Z}^{k+t}$ satisfies
\begin{align}\nonumber
 \bar{h}_i &=  h_i,  \ \forall i=1,\dots,k-1, \\ \nonumber
\bar{h}_{k+j-1} &=  h_k,  \ \forall j=1,\dots,t+1. 
\end{align}
Then 
\[
\sum_{i=1}^{k+t}\bar{h}_iy^i 
= \sum_{i=1}^{k-1}h_iy^i +
 \sum_{j=1}^{t+1}h_ky^{k+j-1} 
= \sum_{i=1}^{k-1}h_ix^i +h_kx^k 
= 0. 
\]
Therefore $\bar{h}$ is an integer relation of circuits $\{y^i\}_{i=1}^{k+t}$. 

Next, we show that $\bar{h}$ is primitive. 
Suppose that  
$h'\in{\mathbb Z}^{k+t}$
is a nontrivial relation on the 
$\{y^i\}_{i=1}^{k+t}$. 
Without loss of generality we may assume that the  
$\{h'_i\}_{i=1}^{k+t}$ are relatively prime integers,
at least one of which is positive. 
We look at the row of $v_{t+1}$. Then it follows that
\begin{equation}\nonumber
 h'_{k}=h'_{k+1}=\dots=h'_{k+t}. 
\end{equation}
Therefore
\[
0=\sum_{i=1}^{k+t}h'_iy^i 
= \sum_{i=1}^{k-1}h'_iy^i +
 h'_k\sum_{j=1}^{t+1}y^{k+j-1} 
= \sum_{i=1}^{k-1}h'_ix^i + h'_kx^k. 
\]
This is an integer relation on $\{x^i\}_{i=1}^k$, and because $h$ is  primitive,
\[
 h'_i =  h_i,  \ \forall i=1,\dots,k . 
\]
Therefore $h'=\bar{h}$ and $\bar{h}$ is primitive. 

Now apply to $\{y^i\}_{i=1}^{k+t}$ a permutation of columns
so that $y^{k+t}$ becomes  $(v_1,u_1,v_2,u_2,\dots,v_{t+1},u_{t+1})$. 
For $i=1,\dots,k+t$, let $\bar{x}^i$ be the circuit of
$A_{t+1,t+1}$ which is the image of $y^i$ under this permutation. 
Then $\{\bar{x}^i\}_{i=1}^{k+t}$ also satisfy the primitive relation
$\sum_{i=1}^{k+t}\bar{h}_i\bar{x}^i=0$ with the same coefficients 
$\bar{h}$. This completes the proof. 
\end{proof}


\begin{Lemma}\label{rtor1}
Let $r\geq t \geq 4$. 
Suppose that there are circuits 
$\{x^i\}_{i=1}^k$  of  $A_{t,r}={\bf 1}_t^{(r)}$
admitting a primitive relation $h$, where the $k$-th circuit and the $k$-th
coefficient are
\begin{align}\nonumber
 x^k&=(v_1,u_{r-t+1},v_2,u_{r-t+2},\dots,v_t,u_{r}),\\ \nonumber
h_k&=1. 
\end{align}
Then there are circuits $\{\bar{x}^i\}_{i=1}^{k+t-1}$
of 
$A_{t,r+1}={\bf 1}_t^{(r+1)}$
admitting primitive relation 
$\bar{h}$, where the $(k+t-1)$-th circuit is 
\[
\bar{x}^{k+t-1} = (v_1,u_{r-t+2},v_2,u_{r-t+3},\dots,v_t,u_{r+1})
\]
and the elements of $\bar{h}$ are
\[
 \bar{h}_{i}=(t-1)h_i, \ \forall i=1,\dots,k-1,\\  \ 
 \bar{h}_k= \bar{h}_{k+1}= \dots= \bar{h}_{k+t-1}= h_k = 1. 
\]
\end{Lemma}

\begin{proof}
Using the natural embedding of 
 $K_{t,r}$ into $K_{t,r+1}$,
we can interpret circuits of the former also as circuits of the latter. 
Put
\begin{equation}\nonumber
 y^i=x^i, \ \forall i=1,\dots,k-1,
\end{equation}
and for all 
$j=1,\dots,t$,  let $y^{k+j-1}$ denote vectors obtained by
changing vertex $u_{r-j+1}$ of
$x^k$ to $u_{r+1}$. 
Table \ref{rtor1mat} displays these circuits as matrices. 
Here for each $i=1,\dots,t$, 
$q^i\in{\mathbb Z}^{t}$ denotes a vector satisfying
\begin{equation}\nonumber
 q^i_i=1,\ q^i_{i+1}=-1,  
\end{equation}
and the rest are zeros. Here we identify  $t+1$ with $1$.

\begin{table}[htbp]
\caption{Circuits for recursion on $r$}
\vspace*{-3mm}
\begin{center}
\setlength{\tabcolsep}{1pt}
\begin{tabular}{c|ccccccc|} 
 & $\ \dots$ & $u_{r-t+1}$ & $u_{r-t+2}$ & $u_{r-t+3}$ & $\dots$ & $u_{r}$ & $u_{r+1}$ \\\hline
$y^{k}=(v_1,u_{r+1},v_2,u_{r-t+2},\dots,v_{t-1},u_{r-1},v_t,u_r)=$ & $\dots$ & $0$ & \multicolumn{1}{c}{$q^2$} & \multicolumn{1}{c}{$q^3$} & $\dots$ & $q^{t}$ & $q^1$ \\\hline
$y^{k+1}=(v_1,u_{r-t+1},v_2,u_{r+1},\dots,v_{t-1},u_{r-1},v_t,u_r)=$  & $\dots$ & $q^1$ & \multicolumn{1}{c}{$0$} & \multicolumn{1}{c}{$q^3$} & $\dots$ & $q^{t}$ & $q^2$ \\\hline
$\vdots$ &  &  & \multicolumn{1}{c}{} & \multicolumn{1}{c}{} & $\vdots$ &  &  \\\hline
$y^{k+t-2}=(v_1,u_{r-t+1},v_2,u_{r-t+2},\dots,v_{t-1},u_{r+1},v_t,u_r)=$ & $\dots$ & $q^1$ & \multicolumn{1}{c}{$q^2$} & \multicolumn{1}{c}{$q^3$} & $\dots$ & $q^{t}$ & $q^{t-1}$ \\\hline
$y^{k+t-1}=(v_1,u_{r-t+1},v_2,u_{r-t+2},\dots,v_{t-1},u_{r-1},v_t,u_{r+1})=$ & $\dots$ & $q^1$ & \multicolumn{1}{c}{$q^2$} & \multicolumn{1}{c}{$q^3$} & $\dots$ & $0$ & $q^{t}$ \\\hline
\end{tabular}
\end{center}
\label{rtor1mat}
\end{table}

%
Notice that 
\begin{equation}\nonumber
\sum_{j=1}^{t} y^{k+j-1} = (t-1)x^k. 
\end{equation}
Define 
\begin{align}\nonumber
 \bar{h}_i &= (t-1) h_i,  \ \forall i=1,\dots,k-1, \\ \nonumber
\bar{h}_{k+j-1} &=  h_k = 1,  \ \forall j=1,\dots,t+1.
\end{align}
Then
\[
\sum_{i=1}^{k+t-1}\bar{h}_iy^i 
= \sum_{i=1}^{k-1}(t-1)h_iy^i +
\sum_{j=1}^{t}h_ky^{k+j-1} 
= \sum_{i=1}^{k-1}h_ix^i +h_kx^k 
= 0. 
\]
Therefore $\bar{h}$ is an integer relation on circuits 
$\{y^i\}_{i=1}^{k+t-1}$. 

Next we show that $\bar{h}$ is  primitive. 
Suppose that 
$h'\in{\mathbb Z}^{k+t-1}$
is a nontrivial relation on the 
$\{y^i\}_{i=1}^{k+t-1}$. 
Without loss of generality we may assume that 
$\{h'_i\}_{i=1}^{k+t-1}$ are relatively prime integers,
at least one of which is positive. 
Consider the column of
$u_{r+1}$.  Then
\begin{equation}\nonumber
 h'_{k}=h'_{k+1}=\dots=h'_{k+t-1}. 
\end{equation}
Therefore
\[
0=\sum_{i=1}^{k+t-1}h'_iy^i
= \sum_{i=1}^{k-1}h'_iy^i +
 h'_k\sum_{j=1}^{t}y^{k+j-1} 
= \sum_{i=1}^{k-1}h'_ix^i + (t-1)h'_kx^k. 
\]
This is an integer relation on $\{x^i\}_{i=1}^k$. 
Therefore there exists  $\alpha\in{\mathbb Z}$ such that
\begin{align}\label{31}
 {h}'_i &= \alpha h_i, \ \forall i=1,\dots,k-1, \\ 
\label{32}
(t-1)h'_k &= \alpha h_k=\alpha. 
\end{align}
Since $h_i>0$ for all $i$ and there is an $i$ such that $h'_i>0$, 
equations (\ref{31}) and (\ref{32}) imply 
$\alpha>0$. 
Therefore (\ref{31}) and (\ref{32}) 
imply that $h'_i>0$ for all
$i$. 
Hence $\bar{h}$ is primitive. 

Now apply to $\{y^i\}_{i=1}^{k+t-1}$ a permutation of columns
so that $y^{k+t-1}$ becomes 
$(v_1,u_{r-t+2},v_2,u_{r-t+3},\allowbreak\dots,v_t,u_{r+1})$. 
For $i=1,\dots,k+t-1$, let $\bar{x}^i$ be the circuit of
$A_{t,r+1}$ which is the image of $y^i$ under this permutation. 
Then $\{\bar{x}^i\}_{i=1}^{k+t-1}$ also satisfy the primitive relation
$\sum_{i=1}^{k+t-1}\bar{h}_i\bar{x}^i=0$ with the same coefficients 
$\bar{h}$. This completes the proof. 
\end{proof}

We are now ready to prove Theorem \ref{thm:main}.
In the proof we use the following notation.
Let ${\mathscr A}(\{x^i\}_{i=1}^k)=\{\bar{x}^i\}_{i=1}^{k+t}$ and
${\mathscr B}(h)=\bar{h}=(\bar{h}_1,\dots,\bar{h}_{k+t-1},1)$ denote 
circuits of $A_{t+1,t+1}$ and the primitive relation  which are
obtained by the operation 
of Lemma \ref{ttot1} to
circuits $\{x^i\}_{i=1}^k$ of $A_{t,t+1}$  and  the primitive relation $h$. 
Note that 
$\| {\mathscr B}(h) \|_1
=\| h \|_1  + t$. 
Furthermore let ${\mathscr A}'(\{x^i\}_{i=1}^k)=\{\bar{x}^i\}_{i=1}^{k+t-1}$ and
${\mathscr B}'(h)=\bar{h}=(\bar{h}_1,\dots,\bar{h}_{k+t-2},1)$ denote 
circuits of $A_{t,r+1}$ and the primitive relation 
which are obtained by the operation 
of Lemma \ref{rtor1} to
circuits $\{x^i\}_{i=1}^k$ of $A_{t,r}$  and  the primitive relation $h$. 
Note that 
$\| {\mathscr B}'(h) \|_1
=(t-1)(\| h \|_1 - 1) + t$. 

Our proof uses induction on $t,r$.
We will construct a primitive relation $h^{(t \times r)}$ on circuits 
${\mathscr X}_{(t \times r)}$ of $A_{t,r}$ by induction.
Therefore we obtain $g(A_{t,r})\geq \|h^{(t \times r)}\|_1$.
Our induction is illustrated in Figure \ref{inductionfigure}.
There,  a down arrow corresponds to the operation of Lemma \ref{ttot1}, and
 a right arrow corresponds to the operation of Lemma \ref{rtor1}.

\begin{figure}[htbp]
\centering
\begin{tabular}{ccccccc} 
$\| h ^{(3\times 4)}\|_1 $ & $\rightarrow$ &  $\| h ^{(3\times 5)}\|_1$ & $\rightarrow$ & $\| h ^{(3\times 6)}\|_1$  & $ \rightarrow $ & $ \cdots $\\
$\downarrow$ &  &  &  &  &  &  \\
$\| h ^{(4\times 4)}\|_1 $ & $ \rightarrow $ & $ \| h ^{(4\times 5)}\|_1 $ & $ \rightarrow $ & $ \| h ^{(4\times 6)}\|_1 $ & $ \rightarrow $ & $ \cdots$ \\
 &  & $\downarrow$ &   &  &  &  \\
 &  & $\| h ^{(5\times 5)}\|_1 $ & $ \rightarrow $ & $ \| h ^{(5\times 6)}\|_1 $ & $ \rightarrow $ & $ \cdots$ \\
 &  &  &  & $\downarrow $ &   &  \\
 &  &  &  & $\vdots $ &   &  \\
\end{tabular}
\caption{Induction on $t,r$}
\label{inductionfigure}
\end{figure}

\begin{proof}[Proof of Theorem \ref{thm:main}]
By induction on  $t$ we will prove   that 
for all $t\geq 4$
there exist  $k(t)=t^2-2t+2$
circuits 
${\mathscr X}_{(t \times t)}=\{x_{(t \times t)}^i\}_{i=1}^{k(t)}
\subset{\mathcal C}(A_{t,t})$
and the primitive relation
$h^{(t \times t)}$
such that
\begin{align}\nonumber
  x_{(t\times t)}^{k(t)}&=(v_1,u_1,v_2,u_2,\dots,v_t,u_t) ,\\ \nonumber
 \sum_{i=1}^{k(t)}h^{(t \times t)}_ix^i&=0,\\ \nonumber
h^{(t \times t)}_{k(t)}&=1,\\ \nonumber
 \| h ^{(t \times t)} \|_1&=
(t-2)!\left(15+\sum_{i=1}^{t-4}\frac{i+4}{(i+2)!}\right). 
\end{align}

Exchange  $x^1$ and $x^7$ of circuits of Example \ref{seven} and 
apply to the circuits  
a permutation of vertices so  that
\begin{equation}\nonumber
 x^7=(v_1,u_2,v_2,u_3,v_3,u_4). 
\end{equation}
Let 
${\mathscr X}_{(3\times 4)}=\{x_{(3\times 4)}^i\}_{i=1}^7$
be the image of 
$\{x^i\}_{i=1}^7$
under this permutation. 
The primitive relation $h^{(3\times 4)}$ on these circuits 
satisfy
\begin{equation}\nonumber
 h^{(3\times 4)}=(7,2,3,3,5,6,1). 
\end{equation}
Notice that $h^{(3\times 4)}_7=1$ holds. 

Let ${\mathscr X}_{(4 \times 4)}={\mathscr A}({\mathscr X}_{(3 \times 4)})$ and
$h^{(4 \times 4)}={\mathscr B}(h^{(3 \times 4)})$
denote the image of
${\mathscr X}_{(3 \times 4)}$
and $h^{(3\times 4)}$
under the operation of Lemma \ref{ttot1}.  Then we have
$h^{(4\times 4)}=(7,2,3,3,5,6,1,1,1,1)\in{\mathbb Z}^{10}$
and 
\begin{align}\nonumber
 x_{(4 \times 4)}^{10}&=(v_1,u_1,v_2,u_2,v_3,u_3,v_4,u_4) ,\\ \nonumber
h^{(4 \times 4)}_{10} &=1,\\ \nonumber
 \| h ^{(4 \times 4)} \|_1 &=
 \| h ^{(3 \times 4)} \|_1   + 3 = 30 .
\end{align}
Therefore we have verified the initial condition at $t=4$ 
for the induction.

Suppose now that the result holds for $t\geq 4$.
Let ${\mathscr X}_{(t \times (t+1))}={\mathscr A}'({\mathscr X}_{(t \times t)})$ and 
$h^{(t \times (t+1))}={\mathscr B}'(h^{(t \times t)})$ 
denote the image of
 ${\mathscr X}_{(t \times t)}$ and $h^{(t \times t)}\in{\mathbb Z}^{k(t)}$
under the operation of Lemma \ref{rtor1}. 
\begin{align}\nonumber
 x_{(t \times (t+1))}^{k(t)+t-1}&=(v_1,u_2,v_2,u_3,\dots,v_t,u_{t+1})
  ,\\ \nonumber
h^{(t \times (t+1))}_{k(t)+t-1} &=1,\\ \nonumber
 \| h ^{(t \times (t+1))} \|_1 &=
(t-1)(\| h ^{(t \times t)} \|_1 -1 )   + t 
\end{align}
follows from Lemma \ref{rtor1}.
Now let ${\mathscr X}_{((t+1) \times (t+1))}={\mathscr A}({\mathscr X}_{(t \times (t+1))})$ and 
$h^{((t+1) \times (t+1))}={\mathscr B}(h^{(t \times (t+1))})$ 
denote the image of
${\mathscr X}_{(t \times (t+1))}$ and $h^{(t \times (t+1))}$
under the operation of Lemma \ref{ttot1}.   Then
\begin{align}\nonumber
 x_{((t+1) \times (t+1))}^{k(t)+2t-1}
&=(v_1,u_1,v_2,u_2,\dots,v_{t+1},u_{t+1}) ,\\ \nonumber
h^{((t+1) \times (t+1))}_{k(t)+2t-1} &=1,\\ \nonumber
 \| h ^{((t+1) \times (t+1))} \|_1 &=
(t-1)(\| h ^{(t \times t)} \|_1 -1 )   + 2t\\ \nonumber
 &= (t-1)!\left(15+\sum_{i=1}^{t-4}\frac{i+4}{(i+2)!}\right)+t+1\\ \nonumber
 &= ((t+1)-2)!\left(15+\sum_{i=1}^{(t+1)-4}\frac{i+4}{(i+2)!}\right)
\end{align}
follows from Lemma \ref{ttot1}.
Here $k(t+1)=k(t)+2t-1$ and $k(4)=10$ imply $k(t)=t^2-2t+2$.
Therefore the result holds for $t+1$.
Henceforth, let $b_t= \| h ^{(t \times t)} \|_1 $.

We fix 
$t\geq 4$
arbitrarily.
We prove by induction on $r$ that, for all $r\geq t$, 
there are circuits 
${\mathscr X}_{(t \times r)}=\{x_{(t \times r)}^i\}_{i=1}^{k(t)}
\subset{\mathcal C}(A_{t,r})$
and the primitive relation
$h^{(t \times r)}$
such that
\begin{align}\nonumber
  x_{(t\times r)}^{k(t)}&=(v_1,u_{r-t+1},v_2,u_{r-t+2},\dots,v_t,u_r) ,\\ \nonumber
\sum_{i=1}^{k(t)}h^{(t \times r)}_ix_{(t\times r)}^i&=0
 ,\\ \nonumber
h^{(t \times r)}_{k(t)}&=1,\\ \nonumber
 \| h ^{(t \times r)} \|_1&=
(t-1)^{r-t}\left(b_t+\frac{1}{t-2}\right)-\frac{1}{t-2}.
\end{align}
The initial condition of the induction, at
$r=t$, follows from 
$ \| h ^{(t \times t)} \|_1 = b_t$.

Suppose now that the result holds for some 
$r\geq t$.
Let ${\mathscr X}_{(t \times (r+1))}={\mathscr A}'({\mathscr X}_{(t \times r)})$ and 
$h^{(t \times (r+1))}={\mathscr B}'(h^{(t \times r)})$ 
denote the image of 
${\mathscr X}_{(t \times r)}$ and $h^{(t \times r)}$
under the operation of Lemma \ref{rtor1}.  
Then
\begin{align}\nonumber
 x_{(t \times (r+1))}^{k(t)+t-1}&=(v_1,u_{r-t+2},v_2,u_{r-t+3},
\dots,v_t,u_{r+1}) ,\\ \nonumber
h^{(t \times (r+1))}_{k(t)+t-1} &=1,\\ \nonumber
 \| h ^{(t \times (r+1))} \|_1 &=
(t-1)(\| h ^{(t \times r)} \|_1 -1 )   + t \\ \nonumber
&=
(t-1)\left((t-1)^{r-t}\left(b_t+\frac{1}{t-2}\right)-\frac{1}{t-2}-1\right)+t\\ \nonumber
&=
(t-1)^{r+1-t}\left(b_t+\frac{1}{t-2}\right)-\frac{1}{t-2}
\end{align}
follows from Lemma \ref{rtor1}.
Therefore the result holds for $r+1$ and
\begin{equation}\nonumber
 g(A_{t,r})\geq (t-1)^{r-t}(b_t+\frac{1}{t-2})-\frac{1}{t-2}
\end{equation}
follows from  Lemma \ref{prop23}. 
\end{proof}

\section{Discussion}
\label{sec:discussion}

In this paper we provided a lower bound in Theorem \ref{thm:main} 
by the induction on $t,r$. Here we discuss some ideas for improving
our lower bound.

Look at Figure \ref{inductionfigure} again.
On the step
$\| h ^{(3\times 4)}\|_1 \rightarrow \| h ^{(4\times 4)}\|_1$, 
we can construct a larger primitive relation 
than the relation constructed  in the proof. 

\begin{Example}
\label{a44}
Let $\{x^i\}_{i=1}^7$ denote the circuits in Example \ref{seven}.
Using the natural embedding of 
$K_{3,4}$
into $K_{4,4}$, let
\begin{equation}\nonumber
 \bar{x}^i=x^i, \   \forall i=1,\dots,6.
\end{equation}
and for $i=7,\dots,10$, we define $\bar{x}^i $ as shown in
Table \ref{7to10}.
\begin{table}[htbp]
\caption{Circuits of $A_{4,4}$}
\begin{center}
\begin{tabular}{c|cccc|c} 
 & $u_1$ & $u_2$ & $u_3$ & $u_4$ &  \\\hline
 & $0$ & $1$ & $0$ & $-1$ & $v_1$ \\
$\bar{x}^7=(v_1,u_2,v_4,u_1,v_3,u_4)=$ & $0$ & $0$ & $0$ & $0$ & $v_2$ \\
 & $-1$ & $0$ & $0$ & $1$ & $v_3$ \\
 & $1$ & $-1$ & $0$ & $0$ & $v_4$ \\\hline
 & $0$ & $0$ & $0$ & $0$ & $v_1$ \\
$\bar{x}^8=(v_2,u_3,v_4,u_2)=$ & $0$ & $-1$ & $1$ & $0$ & $v_2$ \\
 & $0$ & $0$ & $0$ & $0$ & $v_3$ \\
 & $0$ & $1$ & $-1$ & $0$ & $v_4$ \\\hline
 & $0$ & $0$ & $0$ & $0$ & $v_1$ \\
$\bar{x}^9=(v_3,u_4,v_4,u_3)=$ & $0$ & $0$ & $0$ & $0$ & $v_2$ \\
 & $0$ & $0$ & $-1$ & $1$ & $v_3$ \\
 & $0$ & $0$ & $1$ & $-1$ & $v_4$ \\\hline
 & $0$ & $1$ & $0$ & $-1$ & $v_1$ \\
$\bar{x}^{10}=(v_1,u_2,v_2,u_3,v_3,u_1,v_4,u_4)=$ & $0$ & $-1$ & $1$ & $0$ & $v_2$ \\
 & $1$ & $0$ & $-1$ & $0$ & $v_3$ \\
 & $-1$ & $0$ & $0$ & $1$ & $v_4$ \\\hline
\end{tabular}
\end{center}
\label{7to10}
\end{table}
Then $\{\bar{x}^i\}_{i=1}^{10}$
are circuits of ${\rm ker}_{\mathbb{Z}}({\bf 1}_4^{(4)})$
and 
$\bar{h}=(2,4,6,6,10,12,7,7,7,7)$
is its primitive relation.
Then, by Proposition \ref{prop23},
\begin{equation}\label{evl44}
g({\bf 1}_4^{(4)})
\geq
2+4+6+6+10+12+7+7+7+7=68.
\end{equation}
Equation (\ref{evl44}) is sharper than 
the evaluation in Theorem \ref{thm:main}.
\end{Example}

We could start induction from circuits and its primitive relation in 
Example \ref{a44}. Then we obtain a sharper evaluation for some small
$t,r$.  However, unfortunately it turns out that,
if we start from Example \ref{a44}  then on  the step 
$\| h ^{(8\times r)}\|_1 \rightarrow \| h ^{(8\times (r+1))}\|_1$,
we can not obtain an exponential lower bound.  
Therefore we did not use Example \ref{a44} in the proof 
of Theorem \ref{thm:main}.  However this example suggests that there may be some other better
initial set of circuits for our induction.

\bibliographystyle{plainnat}
\bibliography{graver}

\begin{thebibliography}{7}
\providecommand{\natexlab}[1]{#1}
\providecommand{\url}[1]{\texttt{#1}}
\expandafter\ifx\csname urlstyle\endcsname\relax
  \providecommand{\doi}[1]{doi: #1}\else
  \providecommand{\doi}{doi: \begingroup \urlstyle{rm}\Url}\fi

\bibitem[Aoki and Takemura(2003)]{aoki03}
Satoshi Aoki and Akimichi Takemura.
\newblock Minimal basis for a connected {M}arkov chain over {$3\times 3\times
  K$} contingency tables with fixed two-dimensional marginals.
\newblock \emph{Aust. N. Z. J. Stat.}, 45\penalty0 (2):\penalty0 229--249,
  2003.
\newblock ISSN 1369-1473.

\bibitem[Berstein and Onn(2009)]{berstein09}
Yael Berstein and Shmuel Onn.
\newblock The {G}raver complexity of integer programming.
\newblock \emph{Ann. Comb.}, 13\penalty0 (3):\penalty0 289--296, 2009.
\newblock ISSN 0218-0006.
\newblock \doi{10.1007/s00026-009-0029-6}.

\bibitem[De~Loera and Onn(2006)]{loera06}
Jes{\'u}s~A. De~Loera and Shmuel Onn.
\newblock All linear and integer programs are slim 3-way transportation
  programs.
\newblock \emph{SIAM J. Optim.}, 17\penalty0 (3):\penalty0 806--821, 2006.
\newblock ISSN 1052-6234.
\newblock \doi{10.1137/040610623}.

\bibitem[De~Loera et~al.(2008)De~Loera, Hemmecke, Onn, and Weismantel]{loera08}
Jes{\'u}s~A. De~Loera, Raymond Hemmecke, Shmuel Onn, and Robert Weismantel.
\newblock {$n$}-fold integer programming.
\newblock \emph{Discrete Optim.}, 5\penalty0 (2):\penalty0 231--241, 2008.
\newblock ISSN 1572-5286.
\newblock \doi{10.1016/j.disopt.2006.06.006}.

\bibitem[Ho{\c{s}}ten and Sullivant(2007)]{hosten07}
Serkan Ho{\c{s}}ten and Seth Sullivant.
\newblock A finiteness theorem for {M}arkov bases of hierarchical models.
\newblock \emph{J. Combin. Theory Ser. A}, 114\penalty0 (2):\penalty0 311--321,
  2007.
\newblock ISSN 0097-3165.
\newblock \doi{10.1016/j.jcta.2006.06.001}.

\bibitem[Santos and Sturmfels(2003)]{santos03}
Francisco Santos and Bernd Sturmfels.
\newblock Higher {L}awrence configurations.
\newblock \emph{J. Combin. Theory Ser. A}, 103\penalty0 (1):\penalty0 151--164,
  2003.
\newblock ISSN 0097-3165.
\newblock \doi{10.1016/S0097-3165(03)00092-X}.

\bibitem[Sturmfels(1996)]{MR1363949}
Bernd Sturmfels.
\newblock \emph{Gr\"obner Bases and Convex Polytopes}.
\newblock American Mathematical Society, Providence, Rhode Island, 1996.
\newblock ISBN 0-8218-0487-1.

\end{thebibliography}

\end{document}